\documentclass{amsart}

\usepackage{amsmath,amssymb,amsthm}
\usepackage{mathrsfs}
\usepackage{amscd}
\usepackage[active]{srcltx}
\usepackage{verbatim}
\usepackage{enumerate}
\usepackage{graphics}
\usepackage{graphicx}
\usepackage[colorlinks,linkcolor={blue},citecolor={blue},urlcolor={red}]{
hyperref}

\allowdisplaybreaks

\theoremstyle{plain}
\newtheorem{theorem}{Theorem}[section]
\theoremstyle{remark}
\newtheorem{remark}[theorem]{Remark}

\theoremstyle{plain}

\newtheorem{lemma}[theorem]{Lemma}

\numberwithin{equation}{section}

\def\Z{{\mathbb Z}}

\def\R{{\mathbb R}}
\def\C{{\mathbb C}}

\renewcommand{\P}{{\mathbb P}}

\renewcommand{\a}{\alpha}

\newcommand{\g}{\gamma}

\newcommand{\e}{\varepsilon}

\newcommand{\B}{\mathscr{B}}

\newcommand{\n}{\Vert}

\DeclareMathOperator{\dist}{dist}
\DeclareMathOperator{\diam}{diam}

\newcommand{\al}{{a}}
\newcommand{\rr}{{A}}

\renewcommand{\tilde}{\widetilde}

\newcommand{\beq}{\begin{equation}}
\newcommand{\eeq}{\end{equation}}
\newcommand{\bal}{\begin{aligned}}
\newcommand{\eal}{\end{aligned}}
\newcommand{\ben}{\begin{enumerate}}
\newcommand{\een}{\end{enumerate}}
\newcommand{\bit}{\begin{itemize}}
\newcommand{\eit}{\end{itemize}}

\newcommand{\book}[1]{{\em #1}}
\newcommand{\journal}[1]{{\em #1}}
\newcommand{\volume}[1]{{\bfseries #1}}
\newcommand{\name}[1]{{\sc #1}}

\begin{document}

\author{Jan Maas}
\address{
Institute for Applied Mathematics\\
University of Bonn\\
Endenicher Allee 60\\
53115 Bonn\\
Germany}
\email{maas@iam.uni-bonn.de}

\author{Jan van Neerven}
\address{
Delft Institute of Applied Mathematics\\
Delft University of Technology
\\P.O. Box 5031\\
2600 GA Delft\\
The Netherlands} \email{J.M.A.M.vanNeerven@tudelft.nl}

\author{Pierre Portal}
\address{Universit\'e Lille 1, Laboratoire Paul Painlev\'e, 59655 Villeneuve
d'Ascq, France}
\email{pierre.portal@math.univ-lille1.fr}

\thanks{The first named author is supported by Rubicon subsidy 680-50-0901
of the Netherlands Organisation for Scientific Research (NWO). The second named
author is supported by VICI subsidy 639.033.604
of the Netherlands Organisation for Scientific Research (NWO)}

\subjclass[2000]{42B25,42B30}
\keywords{Hardy spaces, Gaussian measure, Ornstein-Uhlenbeck operator, square function, maximal function}

\title[Conical square functions in $L^{1}(\gamma)$]{Conical square
functions and non-tangential maximal functions with respect to the gaussian measure}
\begin{abstract}
We study, in $L^{1}(\R^n;\gamma)$ with respect to the gaussian measure,
non-tan\-gen\-ti\-al maximal functions
and conical square functions associated with the Ornstein-Uhlenbeck operator
by developing a set of techniques  which allow us, to some extent,
to compensate for the non-doubling character of the gaussian measure.
The main result asserts that conical square functions can be controlled in $L^1$-norm by
non-tangential maximal functions. Along the way we prove a change of aperture result for the latter.
This complements recent results on gaussian Hardy spaces due to Mauceri and Meda.
\end{abstract}

\maketitle

\section{Introduction}
Gaussian harmonic analysis, understood as the study of objects associated with
the gaussian measure
$$d\gamma(x) = (2\pi)^{-n/2}\exp(-\tfrac12|x|^2)\,dx$$
on $\R^{n}$, and the Ornstein-Uhlenbeck operator $$Lf(x) =-\Delta f(x)+x\cdot\nabla f(x)$$
on function spaces such as $L^{2}(\R^{n};\gamma)$,
has recently gained new momentum following the development, by Mau\-ce\-ri and
Meda \cite{MM}, of an atomic Hardy space $H_{\rm at}^{1}(\R^{n};\gamma)$, on
which various functions of $L$ give rise to bounded operators.
Harmonic analysis in $L^{p}(\R^{n};\gamma)$ has been relatively well established for
some time, with results such as the boundedness of Riesz transforms going back
to the work of Meyer and Pisier in the 1980's.
The $p=1$ case, however, has always proven to be difficult. Over the last 30
years, some weak type $(1,1)$ estimates have been obtained, while others have
been disproved (see the survey \cite{survey}). The proofs of these results
rely on subtle decompositions and estimates of kernels. Until the seminal
Mauceri-Meda paper appeared in 2007, a large part of euclidean harmonic
analysis, such as end point estimates using Hardy and BMO spaces, seemed to have
no gaussian counterpart.
Gaussian harmonic analysis in $L^{2}(\R^{n};\gamma)$ is relatively straightforward  given the fact
that the Ornstein-Uhlenbeck operator is diagonal with respect to the basis of Hermite polynomials.
The $L^p(\R^{n};\g)$ case,  with $1<p<\infty$, is harder but
still manageable through kernel estimates.
The end points $p=1$ and $p=\infty$, however, usually require techniques such as
Whitney coverings and Calder\'on-Zygmund decompositions, for which the
non-doubling
nature of the gaussian measure, has,
so far, not been overcome.
Mauceri and Meda's  paper \cite{MM}, though, indicates a possible strategy.
The authors used the notion of
{\em admissible balls} which goes back to the work of Muckenhoupt \cite{Muck69}.
These are balls $B(x,r)$ with the property that $r\leq
a\min(1,\frac{1}{|x|})$ for some fixed admissibility parameter $a>0$. On these
admissible balls, the gaussian measure turns
out to be doubling.
The idea is then to follow classical arguments using admissible balls only.
This is easier said than done.
Indeed, admissible balls are very small when their centre is far away
from the origin, whereas tools such as Whitney decompositions of open sets
require the size of balls to be comparable to their distance to the boundary of
the set, hence possibly very large.
This may be why, although it contains many breakthrough results, Mauceri and Meda's paper
\cite{MM} does not yet give a full theory of $H^{1}$ and $BMO$ spaces for the
gaussian measure. For instance, the boundedness of key operators such as maximal
functions, conical square functions (area integrals), and above all Riesz
transforms, is  still missing.
In fact, while this paper was in its final stages, Mauceri, Meda, and Sj\"ogren
\cite{mms} proved that
Riesz transforms (more precisely {\em some} Riesz transforms, see their paper
for the details) are bounded on the Mauceri-Meda  Hardy space only in
dimension one. This suggests that the `correct' $H^{1}(\R^{n};\gamma)$ space should
be a modification of theirs.

In this paper, we take another step towards a satisfying $H^1(\R^{n};\gamma)$ theory by
studying, in $L^{1}(\R^{n};\gamma)$, non-tangential square functions and maximal functions.
These are gaussian analogues of the sublinear operators which,
in the euclidean setting, are the cornerstones of the real variable theory of $H^1(\R^n)$.
In the gaussian context, non-tangential maximal functions
were first introduced in an unpublished work by Fabes and Forzani,
who studied a gaussian counterpart of the Lusin area integral. Their
$L^p$-boundedness was shown subsequently by Forzani, Scotto, and Urbina
\cite{fsu}.
Our definition is an averaged version of a non-tangential maximal function from
a subsequent paper of Pineda and Urbina \cite{pu}.
The additional averaging adds some technical difficulties,
but experience has shown (see e.g. \cite{hm}) that such averaging can
be helpful in Hardy space theory and its applications (to boundary value
problems, for instance).

Here we prove a change of aperture formula for the maximal
function in the spirit of one of the key estimates of Coifman, Meyer and Stein
\cite{CMS}.
We then show that the non-tangential square function is controlled by the non-tangential
maximal function.
Such estimates are central in Hardy space theory (see for instance \cite{dkpv, fs}).
Thus, the purpose of this article is twofold.
On the one hand, it contributes to the development of dyadic techniques in
gaussian harmonic analysis, i.e., methods and results based on gaussian
analogues of the decomposition of $\R^{n}$ into dyadic cubes, the related
covering lemmas, and the corresponding $H^{1}$ and weak type $(1,1)$ estimates.
This makes the paper technical in nature, but we believe that the techniques
developed here will find more applications, as gaussian harmonic analysis
becomes more geometric and relies less on euclidean (after a change of
variables) estimates of the Mehler kernel.
On the other hand, this article gives some of the results required in the development of a gaussian Hardy space theory. When completed, such a theory will not only be satisfying from a pure harmonic analytic perspective, but it should also be applicable to stochastic partial differential equations (SPDE).
Given the success of Hardy space techniques in deterministic PDE, one can
think that a gaussian analogue would similarly have applications to non-linear SPDE and stochastic boundary value problems.

Now let us state the main result of this paper.
We set $$m(x):=\min\big\{1, \frac1{|x|}\big\}$$ and let
$$\Gamma^{\al} _{x}(\gamma) := \big\{(y,t) \in \R^{n}\times(0,\infty)\colon
|y-x| < t < \al m(x)\big\}$$
denote the {\em admissible cone with parameter $\al > 0$} based at the point
$x\in\R^n$. We denote by $(e^{-tL})_{t \geq 0}$ the Ornstein-Uhlenbeck semigroup acting
on $L^{p}(\R^{n};\gamma)$ for $1 < p < \infty$ (see the survey \cite{survey} and the references therein).
For test functions $u \in C_{\rm c} (\R^n)$ and admissibility parameter  $\al>0$
we consider the \emph{conical square function}
\begin{align*}
S_a u(x) =
 \Big(\int_{\Gamma_x^{a}(\g)} \frac1{\gamma(B(y,t))}|t\nabla
e^{-t^{2}L}u(y)|^{2}\,d\gamma(y)\,\frac{dt}{t}\Big)^\frac12.
\end{align*}
and the {\em non-tangential maximal function}
$$
T^{*}_{\al}u(x):=\underset{(y,t)\in \Gamma ^{\al}_{x}(\gamma)}{\sup}
\Big( \frac{1}{\gamma(B(y, t))} \int_{B(y, t)}
|e^{-t^{2}L}u(z)|^{2}d\gamma(z) \Big)^{\frac{1}{2}}.
$$

The main result of this paper reads as follows:

\begin{theorem} \label{thm:main}
For each $a > 0$ there exists an $a' > 0$ such that the conical square function $S_a$ is controlled by the non-tangential maximal function $T_{a'}^*$, in the sense that
\begin{align*}
  \| S_a u \|_{L^1(\R^{n};\g)} \lesssim\|T^{*}_{a'} u\|_{L^1(\R^{n};\g)}
\end{align*}
with implied constant independent of $u \in C_{\rm c} (\R^n).$ \end{theorem}

By using the truncated cones $\Gamma^{\al}_x$, we are only averaging over admissible balls in the definition of the operators. The idea is, of course, to exploit the doubling property of the gaussian measure on these balls. This makes the operators ``admissible". The reader should notice, however, that they are not local, in the sense that their kernels are not supported in a region of the form $\{(x,y)\in \R^{2n} \;;\; |x-y| \leq \frac{a}{1+|x|+|y|}\}$. Moreover, they can not be written as sums of local operators. This is due to a lack of off-diagonal estimates, that is a crucial difference between the Ornstein-Uhlenbeck semigroup and the heat semigroup.

\subsection*{Acknowledgement}

We are grateful to the anonymous referee for valuable suggestions that lead to simplifications and generalisations of various arguments.

\section{A covering lemma}
In this section we introduce partitions of $\R^{n}$ into ``admissible''  dyadic cubes
and use them to prove a
covering lemma which will be needed later on.

We begin with a brief discussion of admissible balls.
Let
$$m(x) := \min\Big\{1,\frac{1}{|x|}\Big\}, \quad x\in\R^n.$$
For $a>0$ we define
$$\B_{a}:=\big\{B(x,r)\colon  x \in \R^{n}, \
0\le r\le am(x)\big\}.$$
The balls in $\B_{a}$ are said to be {\em admissible at scale $a$}.
It is a fundamental observation of Mauceri and Meda \cite{MM} that admissible
balls enjoy a doubling property:

\begin{lemma}[Doubling property]\label{lem:doubling}
Let $a, \tau > 0.$ There exists a constant $d=d_{a, \tau, n}$, depending
only on $a$, $\tau$, and the dimension $n$, such that if $B_1 = B(c_1,r_1)
\in \B_{a}$ and $B_2 = B(c_2,r_2)$ have non-empty intersection
and $r_2 \leq \tau r_1,$ then
$$ \g(B_2)\le d \g(B_1).$$
\end{lemma}

In particular this lemma implies that for all $a>0$ there exists a constant $d'
= d_a'$ such that for all $B(x,r)\in \mathscr{B}_a$ we have
\begin{align*}
\g(B(x,2r))\le d' \g(B(x,r)).
\end{align*}

The first part of the next lemma, which is taken from \cite{mnp2}, says,
among other things, that if $B(x,r)\in\mathscr{B}_a$ and
$|x-y|<br$, then $B(y,r)\in\mathscr{B}_c$ for some constant $c=c_{a,b}$ which
depends only on $a$ and $b$.

\begin{lemma}
\label{lem:admsym}
 Let $a,b>0$ be given.
\begin{enumerate}
\item[\rm(i)] If $r\le am(x)$ and $|x-y|\leq br$, then $r\le c_{a,b}m(y),$ where
$c_{a,b} := a(1+ab)$.
\item[\rm(ii)] If $|x-y| \leq bm(x)$, then $m(x)\le (1+b)m(y)$ and $m(y)\le
(2+2b)m(x)$.
\end{enumerate}
\end{lemma}

For $k\geq 0$ let $\Delta_k$ be the set of dyadic cubes at scale
$k$, i.e.,
$$ \Delta_k = \{2^{-k}(x+[0,1)^n)\colon x\in \Z^n\}.$$
Following \cite{mnp2}, in the gaussian case we only use cubes whose diameter depends
on another parameter $l$, which keeps track of the distance from the ball
to the origin. More precisely, define the {\em layers}
$$ L_0 = [-1,1)^n, \quad L_{l} = [-2^l, 2^l)^n\setminus [-2^{l-1}, 2^{l-1})^n \
\ (l\ge 1), $$
and define, for $k,l\ge 0$,
$$ \Delta_{k,l}^\gamma =\{ Q\in \Delta_{l+k}\colon Q\subseteq L_l\},
\quad \Delta_{k}^\gamma =\bigcup_{l\ge 0}\Delta_{k,l}^\gamma,
\quad \Delta^\gamma = \bigcup_{k\ge 0} \Delta_k^\gamma.$$ Note
that if $Q\in \Delta_k^\gamma$ with $Q\subseteq L_l$, then its
centre $c_Q$ has norm $2^{l-1} \leq |c_Q| \leq 2^l\sqrt{n} $ and we have
\begin{align}\label{eq:diamQ}
\diam(Q) = 2^{-k-l}\sqrt{n} \le 2^{-k}n\,m(c_Q).
\end{align}

\begin{lemma}\label{lem:cube-ball}
If a ball $B(x,r)\in \B_{a}$ intersects a cube $Q\in \Delta_{0}^{\gamma}$ with center $c_Q$,
then 
\begin{align*}
r \leq  2a(a+n)m(c_{Q}).
\end{align*}
\end{lemma}

\begin{proof}
We consider two cases.
First, if $|c_{Q}| \geq 2(a+n)$, we notice that
$$
r \leq \frac{a}{|x|}
  \leq \frac{a}{|c_{Q}|-(r+nm(c_{Q})/2)}
  \leq \frac{a}{|c_{Q}|-(a+n/2)}
  \leq \frac{2a}{|c_{Q}|}
  = 2am(c_{Q});
$$
in the first inequality we used that $\diam(Q) \leq nm(c_{Q})$ by \eqref{eq:diamQ},
in the second we used that $m(c_Q)\le 1$ and $r\le am(x)\le a$,
the third follows from the assumption we made,
and the final identity follows by noting that $|c_Q|\ge 2n\ge  1$.
Second, if $|c_{Q}| \leq 2(a+n),$ then
together with $1 \leq 2(a+n)$ we obtain  $1 \leq 2(a+n)m(c_{Q})$ and $r \leq a \leq 2 a(a+n) m(c_Q)$.
\end{proof}

We denote by $\a\circ Q$ the cube with the same centre as $Q$ and $\a$ times its
side-length; similar notation is used for balls.
Cubes in $\Delta^\gamma$ enjoy the following doubling property:

\begin{lemma}\label{lem:doubling-cubes}
Let $\alpha > 0.$ There exists a constant $C_{\alpha,n},$ depending only on
$\alpha$ and the dimension $n,$ such that for every cube $Q \in \Delta^\gamma$
we have
\begin{align*}
 \gamma(\alpha \circ Q) \leq C_{\alpha,n} \gamma(Q).
\end{align*}
\end{lemma}

\begin{proof}
Without loss of generality we may assume that $\alpha > 1.$
Let $Q \in \Delta_{k,l}^\gamma$ with center $y$ and side-length $2s.$ Set $B =
B(y,s)$ and note that $B \subseteq Q.$ Moreover, we have $\alpha \circ Q
\subseteq \alpha \sqrt{n} \circ B.$
Since, if $|y| > 1,$
\begin{align*}
  2s = \frac{\diam(Q)}{\sqrt{n}}
     = 2^{-k-l}
 \leq 2^{-l}
 \leq \frac{\sqrt{n}}{|y|}
 = \sqrt{n} m(y),
\end{align*}
and, if $|y| \leq 1,$
\begin{align*}
   2s = 2^{-k-l}
      \leq 1
      \leq \sqrt{n} m(y),
\end{align*}
it follows that $B \in \B_{\sqrt{n}/2}.$ Using the doubling property for
admissible balls from Lemma \ref{lem:doubling} we now obtain
\begin{align*}
  \gamma(\alpha\circ Q)
  \leq \gamma(\alpha \sqrt{n} \circ B)
  \leq C_{\alpha, n} \gamma(B)
  \leq C_{\alpha, n} \gamma(Q).
\end{align*}
\end{proof}

\begin{lemma}\label{lem:cover2}
Let $F\subseteq \R^n$ be a non-empty set, let $a,b,c>0$ be fixed, and let
$$
   O := \{x\in \R^n  \colon  0<d(x,F)\le am(x)\}.
$$
There exists a sequence
$(x_k)_{k\ge 1}$ in $O$ with the following properties:
\ben
\item[\rm(i)] $\displaystyle O\subseteq \bigcup_{k\ge 1} B(x_k, b d(x_k,F))$;
\item[\rm(ii)] $\displaystyle \sum_{k\ge 1} \g(B(x_k, cd(x_k,F))) \lesssim \g(O)$ with
constant depending only on $a$, $b$, $c$, $n$.
\een
\end{lemma}

\begin{proof}
Let $\delta := \min\{\frac{1}{2}, b\}$.
We use a Whitney covering of $$O':= \big\{z\in \R^n \colon 0<d(z,F)<  2am(z)\big\}$$ by disjoint cubes $Q_{k}$ such that
$$
\tfrac14{\delta} d(Q_{k},\complement O') \leq {\rm diam}\,(Q_{k}) \leq \delta d(Q_{k},\complement O'),
$$
(see \cite[VI.1]{littleStein}).
We discard the cubes that do not intersect $O$ and relabel the remaining sequence of cubes as $(Q_k)_{k\ge 1}$ with centers $(c_k)_{k \geq 1}$.
For each $k\ge 1$ pick $x_{k} \in O\cap Q_{k}$.

To check that the balls $B(c_k, \diam(Q_k))$ are admissible, we use the fact that $\delta\le \frac12$ to obtain
$$ |c_k - x_k| \le \tfrac12 {\rm diam}\,(Q_k)
\le \tfrac1{4} d(Q_k, \complement O') \le \tfrac1{4} d(x_k,F)
\le \tfrac14{a}m(x_k).$$
Lemma \ref{lem:admsym}(ii) then shows that $m(x_k) \le (1+\frac{a}{4}) m(c_k)$.
It follows that the balls  $B(c_k, \diam(Q_k))$ are admissible.

Next, $\diam(Q_{k})\le \delta d(Q_{k},\complement O') \le b d(x_k,F)$, so (i) follows from$$
O \subseteq \bigcup _{k\ge 1} Q_{k}  \subseteq \bigcup _{k\ge 1} B(x_{k},\diam(Q_{k}))
 \subseteq \bigcup _{k\ge 1} B(x_{k},b {d(x_{k},F)}).
$$

Towards the proof of (ii), we claim
that for all $x\in O$,
 $$d(x,F)\le 3\max\{1,a\} d(x,\complement O').$$
To prove the claim, we fix $x\in O$ and pick an arbitrary $y\in \complement O'$. Setting $\e := \frac13 \min\{1,\frac1a\}$ we need to prove that
\begin{align*}
|x-y| \ge \e d(x,F).
\end{align*}
From $y\not\in O'$ we know that either $d(y,F)\ge 2am(y)$ or $d(y, F) = 0$.
In the latter case we have $y \in \overline{F}$, hence $\e d(x,F) \leq d(x,F) \leq |x-y|,$
so in what follows we may assume that $d(y,F)\ge 2am(y).$
From $x\in O$ we know that $d(x,F)\le am(x)$. Suppose, for a contradiction, that $|x-y| < \e d(x,F)$.
Then $|x-y|< \e am(x)$ and therefore $m(x)\le (1+\e a)m(y)$ by Lemma \ref{lem:admsym}(ii).
Also, for all $f\in F$ we have
$ |x-y|\ge |y-f| - |f-x| \ge 2am(y)- |f-x|$.
Minimising over $f$, this gives
$|x-y| \ge   2am(y) - d(x,F)$.
Since also $\e d(x,F) > |x-y|$, we find that
$am(y) < \frac 12(1+\e) d(x,F)\le \frac 12(1+\e) am(x)$. It follows that $m(y) < \tfrac12(1+\e) m(x),$
and in combination with the inequality $m(x)\le (1+\e a)m(y)$
we get
$$ 2 <  (1+\e)(1+\e a).$$
On the other hand, recalling that $\e = \frac13 \min\{1,\frac1{a}\}$
we see that $ (1+\e)(1+\e a) \le (1+\frac13)(1+\frac13) = \frac{16}{9} <2.$
This contradicts the previous inequality and the claim is proved.

Combining the estimate
\begin{align*}
 d(x_k, \complement O')
&  \leq  d(Q_k, \complement O') + \diam(Q_k)
  \leq \Big(1 + \frac4{\delta}\Big)  \diam(Q_k)
\end{align*}
with the claim, we obtain
\begin{align*}
 {d(x_{k},F)}
& \le 3\max\{1,a\}d(x_{k},\complement O')
\le 3 \Big(1 + \frac4{\delta}\Big) \max\{1,a\} \diam(Q_{k}).
\end{align*}
Recalling the inequality $|c_k-x_k|\le \frac14 d(x_k,F)$ proved before,
and then using the doubling property in combination with the above inequality,
we obtain
\begin{align*}
\sum_{k\ge 1} \gamma(B(x_{k},c d(x_{k},F))
&\leq \sum_{k\ge 1} \gamma(B(c_{k},(c+\tfrac{1}{4}) d(x_{k},F))
\\&\lesssim \sum_{k\ge 1} \gamma(B(c_{k},\diam(Q_{k}))
 \lesssim \sum_{k\ge 1} \gamma(Q_{k})
\leq \gamma(O').
\end{align*}

To finish the proof we will show that $\g(O')\lesssim \g(O)$
with a constant depending
only on $a$ and $n$.
Using the notation of Lemma \ref{lem:doubling-cubes} it suffices to show that
there exists a sequence of disjoint cubes
$(Q_i)_{i\ge 1} \subseteq \Delta^\g$ contained in $O$
such that $$O'\setminus O \subseteq \bigcup_{i\ge 1} M_{a,n}\circ Q_i$$
with $M_{a,n}$ depending only on $a$ and $n$.
Once this has been shown the claim follows from Lemma
\ref{lem:doubling-cubes}:
$$ \g(O'\setminus O)\le \sum_{i\ge 1} \g(M_{a,n}\circ Q_i)\lesssim \sum_{i\ge 1} \g(Q_i) =
\g\big(\bigcup_{i\ge 1} Q_i\big)
\le \g(O)$$
and consequently $\g(O')\lesssim \g(O)$.

Let $$O'' := \{ x\in \R^n:\ \tfrac13 am(x) < d(x,F) < \tfrac23 am(x)\}.$$
Every point $x\in O''$
belongs to some maximal cube $Q_x\in \Delta^\g$ entirely
contained in $O$. Since any two such maximal cubes are either equal or
disjoint, we may select
a sequence $(x_i)_{i\ge 1}$ in $O''$ such that the maximal cubes $Q_{x_i}\in
\Delta^\g$ are disjoint
and cover $O''$. We will show that these cubes have the desired property.

Fix $y\in O'\setminus O$. Then $d(y,F) = cm(y)$ for some $a\le c< 2a$.
Choose $f\in \overline F$ with $|f-y| = cm(y)$
(this is possible since $\overline F\cap\{z\colon |y-z|\le 2cm(y)\}$ is
compact and non-empty).
Choose $0<\lambda<1$ such that $g:= (1-\lambda)f+\lambda y$ belongs to $O''$.
Choose the index $i$ such that $g\in Q_{x_i}$.

Let $\beta \in (0,\tfrac13)$ be so small that $(\frac23 + \beta)(1+\beta a)<1$. We claim that $0 < d(z,F)< a m(z)$ whenever $|z-g| < \beta a m(g).$
To prove this, note that on the one hand,
\begin{align*}
 d(z, F)
  \geq d(g,F) - |z-g|
  \geq \tfrac13 a m(g) - \beta a m(g)
  > 0,
\end{align*}
while on the other hand, by Lemma \ref{lem:admsym},
\begin{align*}
 d(z, F)
  \leq d(g,F) + |z-g|
  \leq \tfrac23 a m(g) + \beta a m(g)
  \leq (\tfrac23 +\beta) (1 + \beta a) a m(z)
     < a m(z),
\end{align*}
which proves the claim.
It thus follows that $B(g, \beta a m(g))$ is contained in $O$. By maximality, this implies
that the diameter of $Q_{x_i}$ is at least $\frac12 \beta a m(g)$. Hence the side-length $l_i$ of $Q_{x_i}$ is at least $\frac1{2\sqrt n} \beta am(g)$. On the other hand we have
\begin{align*}
|g - y|
  = (1-\lambda) |f - y| < d(y,F) \le
  2am(y)\leq 2a (1+2a)m(g)
\end{align*}
by Lemma \ref{lem:admsym}.
For any $M>1$,  the cube $M\circ Q_{x_i}$ contains the ball
$B(g, \frac12 l_i (M-1))$, and therefore $y\in M\circ Q_{x_i}$ provided $ \frac12 l_i (M-1) \ge 2a(1+2a)m(g).$
This happens if we take $M = 1+8\frac{\sqrt n}{\beta}(1+2a)$, since then
\begin{align*}
\tfrac12 l_i (M-1)
  \geq \frac1{4\sqrt n} \beta am(g)\cdot 8\frac{\sqrt n}{\beta}(1+2a)
  = 2 a (1+2a) m(g).
\end{align*}
It thus follows that $M \circ Q_{x_i}$ contains $y$, which completes the proof.
\end{proof}

\section{Change of aperture for maximal functions}

In the proof of Theorem \ref{thm:main} we need a change of aperture
result
for the admissible cone appearing in the definition of non-tangential maximal
functions. For this purpose we define, for $\rr,\al>0$, the {\em non-tangential maximal function with parameters $\rr, \al $} by
$$
T^{*}_{(\rr, \al)}u(x):=\underset{(y,t)\in \Gamma ^{(\rr, \al)}_{x}(\gamma)}{\sup}
\Big( \frac{1}{\gamma(B(y,\rr t))} \int_{B(y,\rr t)}
|e^{-t^{2}L}u(z)|^{2}d\gamma(z) \Big)^{\frac{1}{2}},
$$
where
$$\Gamma^{(\rr,\al)} _{x}(\gamma) := \Big\{(y,t) \in \R^{n}\times(0,\infty)\colon
|y-x| < \rr t  \textrm{ and }  t < \al m(x)\Big\}$$
is the {\em admissible cone with parameters $\rr, \al$} based at the point
$x\in\R^n$.
The parameter $\rr$ is called the {\em aperture} of the cone.

In what follows we will fix the dimension $n$ and write
$L^p(\g):= L^p(\R^{n};\g)$.

\begin{theorem}[Change of aperture]
\label{thm:aperture}
For all $\rr, A', \al > 0$ there exists a constant $D$, depending only on $\rr$,  $A'$, $\al$,
and the dimension $n$, such that for all $u \in L^{1}(\gamma)$ and $\sigma>0$
we have
$$ \gamma\big(\big\{x \in \R^{n}\colon T^{*}_{(\rr,\al)}u(x)>D\sigma\big\}\big)
\lesssim \gamma\big(\big\{x \in \R^{n}:\
T^{*}_{(A',a')}u(x)>\sigma\big\}\big)
$$
with $a' = \al(1+2\al \rr)(1+ A'\al(1+2\al \rr))$
and with implied constant independent of $u$ and $\sigma$. In particular,
$$
\|T^{*} _{(\rr,\al)} u\|_{L^1(\g)} \lesssim
\|T^{*}_{(A',a')}u\|_{L^1(\g)}
$$
with implied constant independent of $u\in L^1(\gamma)$.
\end{theorem}

The proof of this theorem follows known arguments in the
euclidean case \cite{fs}. We begin with a gaussian weak type $(1,1)$ estimate
from \cite{MM}. For the
convenience of the reader we include an alternative and self-contained
proof.
\begin{lemma}\label{lem:maximal}
Let $a > 0$. For $f\in L_{\rm loc}^1(\R^n)$ put
 $$
 M_a^{*}f(x) := \underset{B(x,r) \in
\B _{a}}{\sup}\frac{1}{\gamma(B(x,r))} \int
_{B(x,r)}|f(y)|\,d\gamma(y).
 $$
Then for all $\tau >0$,
$$\tau \gamma(\{M_a^*(f)> \tau \})
\lesssim \n f\n_{L^1(\g)}$$
with  implied constant only depending
on $a$ and $n$.
\end{lemma}

\begin{proof}
Fix $f\in L_{\rm loc}^1(\R^n)$ and
decompose it as $ f =\sum_{Q\in \Delta_0^\g} 1_Q f$.
We denote by $c_Q$ the centre of a cube $Q$.
The idea of this proof is that the gaussian density is essentially equal to $e^{-\frac12|c_{Q}|^{2}}$ on an admissible ball $B(c_{Q},r_{Q})$ and the support of
$M_{A}^{*}(1_{Q}f)$ is included in such admissible balls.

To make this precise, consider a cube $Q\in \Delta_{0}^{\gamma}$, and suppose that a ball $B(x,r)\in \B_{a}$ intersects $Q$.
Then Lemma \ref{lem:cube-ball} implies that $r \leq  2a(a+n)m(c_{Q}).$ 
As a consequence, for any $y \in B(x,r)$, we use the triangle inequality and \eqref{eq:diamQ} to obtain
\begin{align*}
|y-c_{Q}| \leq 2r + \tfrac12\diam(Q) 
\leq  (4a(a+n) + \tfrac12 n)  m(c_{Q})
  =: b_{a,n} m(c_Q),
\end{align*}
and thus
\begin{align*}
 \big| |c_Q|^2 - |y|^2 \big|
    & \leq  | c_Q + y |\, |c_Q - y|
    \leq \big(2 |c_Q| + | c_Q - y |\big) |c_Q - y|
      \leq  2  b_{a,n}  + b_{a,n}^2.
\end{align*}
This inequality implies that
\begin{align*}
  e^{-\tfrac12|y|^2}    \eqsim e^{-\tfrac12|c_Q|^2}
\end{align*}
with implied constants depending only on $a$ and $n$.

Using this estimate we obtain
\begin{equation}
\label{eq:hl}
\frac{1}{\gamma(B(x,r))} \int  _{B(x,r)} |1_{Q}(y)f(y)| d\gamma(y) \lesssim \frac{1}{|B(x,r)|} \int  _{B(x,r)}|1_{Q}(y)f(y)| dy,
\end{equation}
where $|B(x,r)|$ denotes the Lebesgue measure of the ball; the constants depend only on $a$ and $n$.

Next we note that if $M^{*}_{a}(1_{Q}f)(x)>0$, then
there exists a ball $B \in \B_a$ such that $B \cap Q \neq \emptyset$.
Using (\ref{eq:hl}), we thus have, for all $\tau>0$,
\begin{align*}
\gamma(\{M^{*}_{a}(1_{Q}f)>\tau\})
 \eqsim  \int _{\{M^{*}_{a}(1_{Q}f)>\tau\} } e^{-\tfrac12 |x|^{2}}dx
 & \lesssim  e^{-\tfrac12 |c_{Q}|^{2}}|\{M^{*}_{a}(1_{Q}f)>\tau\}|
\\ & \lesssim e^{-\tfrac12 |c_{Q}|^{2}}|\{ M_{HL}(1_{Q}f) >\tau\}|,
\end{align*}
where $M_{HL}$ denotes the euclidean Hardy-Littlewood maximal operator.
Using the weak type $(1,1)$ bound for the latter, we get
$$
\tau\gamma(\{M^{*}_{a}(1_{Q}f)>\tau\}) \lesssim e^{-\tfrac12 |c_{Q}|^{2}} \int  _{Q} |f(y)|dy \lesssim \|1_{Q}f\|_{L^{1}(\gamma)},
$$
with constants depending only on $a$ and $n$.

Now fix a cube $Q \in \Delta_{0}^{\gamma}$ with centre $c_{Q}$. 
For $x \in Q$ we have $|x - c_Q| \leq \frac12 \diam(Q) \leq \frac12 nm(c_Q)$ by \eqref{eq:diamQ}, and therefore by the second part of Lemma \ref{lem:admsym}(ii), $m(x) \leq 2 (1 + \frac12 n)m(c_Q)$. Hence if $B(x,r) \in \mathcal{B}_a$ and $x\in Q$, then $r \leq a m(x) \leq (2+n)am(c_{Q})$. Thus
\begin{align*}
B(x,r) \subseteq \bigcup _{l=1} ^{N_Q} Q^{{(l)}},
\end{align*}
where we denote by $Q^{{(l)}}$, $l=1,...,N_Q$, the cubes from $\Delta_{0}^{\gamma}$ that satisfy $d(Q,Q^{{(l)}}) \leq (2+n)am(c_{Q})$. Remark that $N_Q\le N$, where $N=N_{a,n}$ only depends on $a$ and $n$.  

It follows from the preceding considerations that
\begin{align*}
M_{a}^{*}f(x) \leq M_{a}^{*}\big(\sum  _{l=1} ^{N_Q} 1_{Q^{{(l)}}}f\big)(x)
\leq \sum  _{l=1} ^{N_Q} M_{a}^{*}(1_{Q^{{(l)}}}f)(x)
\end{align*}
for $x \in Q$, and thus
\begin{equation*}
\begin{split}
\gamma(\{M_{a}^{*}f>\tau\})& = \sum \limits _{Q \in \Delta_{0}^{\gamma}} \gamma(\{M_{a}^{*}f>\tau\}\cap Q)
\leq \sum \limits _{Q \in \Delta_{0}^{\gamma}} \gamma\Big(\Big\{\sum \limits _{l=1} ^{N_Q} M_{a}^{*}(1_{Q^{{(l)}}}f)>\tau\Big\}\Big)\\
& \leq \sum \limits _{Q \in \Delta_{0}^{\gamma}} \sum \limits _{l=1} ^{N_Q}  \gamma\Big(\Big\{M_{a}^{*}(1_{Q^{{(l)}}}f)>\frac{\tau}{N_Q}\Big\}\Big)
\lesssim  \sum \limits _{Q \in \Delta_{0}^{\gamma}}  \frac{N_Q}{\tau}  \sum \limits _{l=1} ^{N_Q}\|1_{Q^{{(l)}}}f\|_{L^{1}(\gamma)}
\\ & \lesssim \frac{N N'}{\tau} \|f\|_{L^{1}(\gamma)},
\end{split}\end{equation*}
with implied constant depending only on $a$ and $n$; the  $N'$ in the last inequality accounts for the fact that, 
given $Q' \in \Delta_{0} ^{\gamma}$, there are at most $N'$ cubes $Q \in \Delta_{0} ^{\gamma}$ such that $\dist(Q,Q') \leq (2+n)am(c_{Q})$,
where again $N'$ depends only on $a$ and $n$.
\end{proof}

\begin{proof}[Proof of Theorem \ref{thm:aperture}]
It suffices to prove the inequality for test functions $u\in C_{\rm c}(\R^n)$. For the rest of
the proof we fix $u\in C_{\rm c}(\R^n)$.
Using the doubling property on admissible balls, we fix a constant $\tau >0$ such that
$$\gamma(B(y,(A'+4 \rr)t)) < \frac{1}{\tau }
\gamma(B(y, A't))\quad \forall B(y,t) \in \B _{c_{\al,2\rr}},$$
where $c_{\al,2\rr} =(1+2\al \rr)\al$ is the constant arising from
Lemma \ref{lem:admsym}(i).
For $\sigma>0$ we define
\begin{align*}
E_{\sigma}     & := \{x \in \R^{n}\colon  T^{*}_{(A',a')}u(x)>\sigma\}, \\
\widetilde{E_{\sigma}} & := \{x \in \R^{n}\colon
M_{b}^{*}(1_{E_{\sigma}})(x)>\tau \},
\end{align*} where
$ M_{b}^{*}f$ is defined as in the lemma and  $b := (A'+2\rr)\al.$
The scheme of the proof is the following. We first prove (step 1) that, if
$ x \not \in \tilde{E}_{\sigma}$ and $(y,t) \in \Gamma_{x}^{(2A,a)}(\gamma)$, then $B(y,A't) \not \subseteq E_{\sigma}$.
We then use this fact (step 2) to prove that
\begin{equation*}\frac{1}{\gamma(B(y,A't))}
\int  _{B(y,A't)} |e^{-t^{2}L}u(\zeta)|^{2}\, d\gamma(\zeta) \leq
\sigma^{2},
\end{equation*}
for all $(y,t) \in \Gamma^{(2\rr,\al)}_{x}(\gamma)$ with $x\not\in \widetilde{E_\sigma}$.
This eventually gives (step 3) that there exists $D=D_{A,A',a,n}>0$ such that
$\{x \in \R^{n} \;;\; T^{*}_{(A,a)}u(x)>D\sigma\} \subseteq \widetilde{E}_{\sigma}$.
The proof is then concluded using Lemma \ref{lem:maximal} applied to
$1_{E_\sigma}$.
In the estimates that follow, the implicit constants are independent of $u$ and $\sigma$.

Throughout steps 1--3 below, we fix a point $x \not \in \widetilde{E_{\sigma}}$ and a point $(y,t) \in \Gamma^{(2\rr,\al)}_{x}(\gamma).$

{\em Step 1} -- We claim that
$B(y,A't) \not \subseteq E_{\sigma}$. To prove this, first note that
from $|x-y|\le 2\rr t$ we have
$$B(y,A't)\subseteq B(x,(A'+2\rr)t)\subseteq
B(y,(A'+4\rr)t).$$
Furthermore, $t\le \al m(x),$
and therefore $B(x,(A'+2\rr)t)\in\mathscr{B}_{(A'+2\rr)\al}
= \mathscr{B}_{b}$.
If we now assume that the claim is false, we get
\begin{align*}
M^{*}_{b}(1_{E_{\sigma}})(x)
&=  \underset{B(x,r) \in \B _{b}}{\sup}\frac{\g(B(x,r)\cap
E_{\sigma})}{\gamma(B(x,r))}
\\ & \ge \underset{B(x,r) \in
\B _{b}}{\sup}\frac{\g(B(x,r)\cap B(y, A't))}{\gamma(B(x,r))}
\\ & \geq \frac{\g(B(x,(A'+2\rr)t)\cap B(y, A't))}{\gamma(B(x,(A'+2\rr)t))}
\\ & = \frac{\gamma(B(y, A't))}{\gamma(B(x,(A'+2 \rr)t))}
\\ & \geq \frac{\gamma(B(y, A't))}{\gamma(B(y,(A'+4\rr)t))}
\\ &  > \tau ,
\end{align*}
where the second inequality uses that $B(x,(A'+2A)t)\in \B _{b}$
and the last one follows from the definition of the constant $\tau $ and
the observation that $B(y,t)\in
\mathscr{B}_{c_{\al,2\rr}}$ by Lemma \ref{lem:admsym}(i), using that
$B(x,t)\in \mathscr{B}_{\al}$ and $|x-y|\le 2\rr t$.
This contradicts the fact that $x\not\in \widetilde{E_\sigma}$ and the claim is
proved.

{\em Step 2} -- Since $B(y, A't) \not \subseteq E_{\sigma}$, there exists $\tilde{y} \in
B(y, A't)$ such that $\tilde y\not\in E_\sigma$, that is,
\begin{align}\label{eq:tilde-y}
\underset{(z,s)\in\Gamma^{(A',a')}_{\tilde{y}}(\gamma)}{\sup}
\frac{1}{\gamma(B(z,A's))}
\int  _{B(z,A's)} |e^{-s^{2}L}u(\zeta)|^{2}\, d\gamma(\zeta) \leq
\sigma^{2}.
\end{align}
Remark also that, since $B(y,t)\in\B _{c_{a,2A}}$ and thus $t \leq c_{\al,2\rr} m(y)$,
Lemma \ref{lem:admsym} implies that $t \leq a' m(\tilde{y})$.
Thus $(y,t)\in \Gamma^{(A',a')}_{\tilde{y}}(\gamma)$ and therefore (\ref{eq:tilde-y}) implies
\begin{equation}\label{eq:est-zeta}
\frac{1}{\gamma(B(y,A't))}
\int  _{B(y,A't)} |e^{-t^{2}L}u(\zeta)|^{2}\, d\gamma(\zeta) \leq
\sigma^{2}.
\end{equation}

{\em Step 3} -- Next let $(w,s)\in \Gamma_x^{(\rr,\al)}(\g)$ be arbitrary and fixed
for the moment.
Then $w \in B(x,\rr s)$. For any $v\in B(w,\rr s)$
we have $|v-x|\le |v-w|+|w-x|\le 2\rr s$. Since also
$s \le \al  m(x)$, it follows that $(v,s) \in \Gamma
^{(2\rr,\al)}_{x}(\gamma)$.
Also, since $|v-w|\le \rr s$
implies  $B(v, A's)\subseteq B(w,(A'+\rr)s)$, we have
$$\g(B(v, A's))\le \g(B(w,(A'+\rr)s))\lesssim \g(B(w,\rr s))$$ by the doubling property for admissible balls; the balls $B(w,\rr s)$ are indeed
admissible by Lemma \ref{lem:admsym}(i).

We can cover ${B(w,\rr s)}$ with finitely many balls of the form $B(v_{i}, A's)$ with $v_i\in B(w,\rr s)$;
this can be achieved with $N =N(\rr,A',n)$ balls. We then have, by
\eqref{eq:est-zeta},
\begin{align*}
\ & \frac{1}{\gamma(B(w,\rr s))} \int  _{B(w,\rr s)} |e^{-s^{2}L}u(z)|^{2}
\,d\gamma(z)
\\ & \qquad \lesssim \sum  _{i=1}^N \frac{1}{\gamma(B(v_{i},A's))} \int
_{B(v_{i},A's)} |e^{-s^{2}L}u(z)|^{2} d\gamma(z) \lesssim \sigma^{2}.
\end{align*}
Taking the supremum over all $(w,s)\in \Gamma_x^{(\rr,\al)}(\g)$, we infer that there exists a constant $D>0$, depending only on
$\rr$, $A'$, $\al$, and the dimension $n$, such that $T^{*}_{(\rr,\al)}u(x) \leq D\sigma$ for all $x \not \in \widetilde{E_{\sigma}}$.

We have now shown that
$ \{T^{*}_{(\rr,\al)}u(x) > D\sigma\} \subseteq \widetilde{E_\sigma}$.
The first assertion of the theorem follows from this via Lemma \ref{lem:maximal}.
The second assertion follows from the first by integration:
$$\begin{aligned}
\|T^{*}_{(\rr,\al)}u\|_{L^1(\g)} & =D \int  _{0} ^{\infty}\gamma(\{x \in \R^{n}\colon
T^{*}_{(\rr,\al)}u(x)>D\sigma\})\,d\sigma
\\ & \lesssim  \int  _{0} ^{\infty}\gamma(\widetilde{E_{\sigma}})\, d\sigma
\lesssim
\int  _{0} ^{\infty}\gamma(E_{\sigma})\,d\sigma
=\|T^{*}_{(A',a')}u\|_{L^1(\g)}.
\end{aligned}
$$
Since the choice of $\rr,A',\al \ge 0$ was arbitrary, this concludes the proof.
\end{proof}

\section{Proof of Theorem \ref{thm:main}}\label{sec:max-quad}

In this section we follow the method pioneered in \cite{fs} for proving square
function estimates in Hardy spaces. This method has recently been adapted in a
variety of contexts (see \cite{amr,ar,hm}). Here, we modify the version given in
\cite{hm} to avoid using the doubling property on non-admissible balls, and to
take into account differences between the Laplace and the Ornstein-Uhlenbeck operators.
As a typical example of the latter phenomenon, we start by proving a gaussian version of the parabolic Cacciopoli inequality.
Recall that $L$ is the Ornstein-Uhlenbeck operator, defined for $f\in C_{\rm b}^2(\R^n)$ by \begin{equation}\label{eq:L}
 Lf(x) =-\Delta f(x)+x\cdot \nabla f(x).
\end{equation}
Note that, for all $f,g \in C_{\rm b}^2(\R^n),$ one has the integration by parts formula
 \begin{align*}
  \int_{\R^n} L f \cdot g \, d\gamma
    =  \int_{\R^n} \nabla f \cdot \nabla g \, d\gamma
 \end{align*}

\begin{lemma}
\label{lem:caccio}
Let $v: \R^{n} \times (0,\infty) \to \C$ be a $C^{1,2}$-function such that $v(\cdot ,t) \in C_{\rm b}^2(\R^n)$ for all $t > 0,$ and suppose that
$$\partial_{t}v+Lv=0$$
on $I(x_{0},t_{0},2r) := B(x_{0},2cr) \times
[t_{0}-4r^{2},t_{0}+4r^{2}]$ for
 some $r \in (0,1),$ $0 < C_0 \leq c \leq C_1 < \infty,$ and $t_{0}>4r^{2}$.
Then
$$
\int  _{I(x_{0},t_{0},r)} |\nabla v(x,t)|^{2}\,d\gamma(x)\,dt \lesssim \frac{1 +
r|x_{0}|}{r^{2}}
\int  _{I(x_{0},t_{0},2r)} |v(x,t)|^{2}\,d\gamma(x)\,dt,
$$
with implied constant depending only on the dimension $n,$ $C_0$ and $C_1$.
\end{lemma}
\begin{proof}
Let $\eta \in C^{\infty}(\R^{n}\times (0,\infty))$ be a cut-off function such
that $0\le \eta\le 1$ on $\R^n\times(0,\infty)$, $\eta\equiv 1$ on
$I(x_{0},t_{0},r)$,
$\eta\equiv 0$ on the complement of $I(x_{0},t_{0},2r)$, and
$$\|\nabla \eta\|_{\infty} \lesssim \frac{1}{r},\quad
\|\partial_{t} \eta\|_{\infty} \lesssim \frac{1}{r^{2}}, \quad
\|\Delta \eta\|_\infty \lesssim \frac1{r^2}$$
with implied constants depending only on $n$, $C_0$, $C_1$. Then,
in view of
  $\n x\cdot \nabla \eta\n_\infty \lesssim (|x_0|+2r)\cdot \frac{1}{r}$
and recalling that $0<r<1$,
\beq\label{eq:Leta}
\|L\eta\|_{\infty}
 \lesssim \frac{1}{r^{2}} + \frac{1}{r} |x_0| + 1
 \lesssim \frac{1 + r|x_{0}|}{r^{2}},
\eeq
where the implied constants depend only on $n,$ $C_0$, $C_1.$

Considering real and imaginary parts separately, we may assume that all
functions are real-valued.
Integrating the identity
$$(\eta\nabla v)\cdot(\eta\nabla v) = (v\nabla \eta-\nabla (v\eta))\cdot
(v\nabla \eta-\nabla (v\eta))$$ and then
using that
\begin{align*}
\int_{I(x_{0},t_{0},2r)}\eta^2 \nabla(v\eta)\cdot \nabla(v\eta)\,d\g\,dt
& \le \int  _{0} ^{\infty} \int_{\R^d}\nabla(v\eta)\cdot
\nabla(v\eta)\,d\g\,dt
\\ & = \int  _{0} ^{\infty}\int_{\R^n}v\eta L(v\eta)\,d\g\,dt
\\ & = \int_{I(x_{0},t_{0},2r)} v\eta L(v\eta)\,d\g\,dt,
\end{align*} we obtain
\begin{equation}\label{eq:threeterms}
\begin{aligned}\int  _{I(x_{0},t_{0},r)}  |\nabla v|^{2}\,d\gamma\,dt
& \le
\int_{I(x_{0},t_{0},2r)} \eta^2|\eta\nabla v|^{2}\,d\gamma\,dt\\
& \leq
\int_{I(x_{0},t_{0},2r)}   \eta^2|v\nabla \eta|^{2}\,d\gamma\,dt\\
& \qquad +
\Big |\int_{I(x_{0},t_{0},2r)} 2v\eta^2\nabla (v\eta)\cdot\nabla \eta\,d\gamma\,dt \Big|\\
& \qquad + \Big |\int_{I(x_{0},t_{0},2r)} v\eta L(v\eta)\,d\gamma\,dt \Big |.
\end{aligned}
\end{equation}
For the first term on the right-hand side we have the estimate
$$ \int_{I(x_{0},t_{0},2r)}   \eta^2|v\nabla \eta|^{2}\,d\gamma\,dt
\lesssim
\frac{1}{r^2}\int_{I(x_{0},t_{0},2r)}|v|^{2}\,d\gamma\,dt.
$$
For the second term we have, by \eqref{eq:Leta},
\begin{align*}
\Big|\int_{I(x_{0},t_{0},2r)}  2v\eta^2 \nabla (v\eta)\cdot \nabla \eta
\,d\gamma\,dt\Big|
& = \frac12\Big|\int_{I(x_{0},t_{0},2r)} \nabla (v\eta)^{2}\cdot\nabla
\eta^{2}\,d\gamma\,dt\Big|\\
& \leq \frac12\Big|\int_{\R^{n}} (v\eta)^2 L\eta^{2}\,d\gamma \,dt\Big| \\
& \lesssim \frac{1 + r|x_0|}{r^{2}}\int  _{I(x_{0},t_{0},2r)}
|v|^{2}\,d\gamma\,dt
\end{align*}
where we used the fact that $\eta^2$ satisfies the same
assumptions as $\eta$ and \eqref{eq:Leta} was applied to $\eta^2$.
To estimate the third term on the right-hand side of \eqref{eq:threeterms} we
substitute the identity
\begin{align*}
L(v\eta) = \eta L v + v L\eta - 2\nabla v \cdot\nabla \eta = -\eta
\partial_{t} v + v L\eta -2 \nabla v \cdot\nabla \eta
\end{align*}
and estimate each of the resulting integrals:
\begin{align*}
\Big|\int_{I(x_{0},t_{0},2r)}  v \eta^{2}\partial_{t}v\,d\gamma\, dt\Big|
& = \frac12\Big|\int_{I(x_{0},t_{0},2r)} \eta^2 \partial_{t}v^2\,d\gamma
\,dt\Big|\\
& = \frac12\Big|\int_{I(x_{0},t_{0},2r)} v^2\partial_{t}\eta^2\,d\gamma
\,dt\Big| \\
& = \Big|\int_{I(x_{0},t_{0},2r)} v^2\eta\partial_{t}\eta\,d\gamma \,dt\Big| \\
& \lesssim\frac{1}{r^{2}}\int  _{I(x_{0},t_{0},2r)} |v|^2\,d\gamma\, dt,\\
\Big|\int_{I(x_{0},t_{0},2r)}  v^2 \eta L\eta\,d\gamma\, dt\Big|
& \lesssim\frac{1 + r|x_0|}{r^{2}}\int  _{I(x_{0},t_{0},2r)} |v|^2\,d\gamma\,
dt, \\
\Big|\int_{I(x_{0},t_{0},2r)}  v\eta \nabla v\cdot \nabla \eta\,d\gamma\,dt\Big|
& = \frac14\Big|\int_{I(x_{0},t_{0},2r)} \nabla v^2\cdot\nabla
\eta^2\,d\gamma\,dt\Big|\\
& = \frac14\Big|\int_{\R^n} v^2 L\eta^2\,d\gamma \,dt\Big| \\
& \lesssim \frac{1 + r|x_0| }{r^{2}}\int  _{I(x_{0},t_{0},2r)}
|v|^2\,d\gamma\,dt.
\end{align*}
\end{proof}

Below we shall apply the lemma with $v(x,t) = e^{-t L} u(x),$ where $u$ is a function in $C_{\rm c}(\R^n)$. From the representation $e^{-t L} u(x) = \int_{\R^n} M_{t}(x,y) u(y)\, dy$ where $M$ is the Mehler kernel (see, e.g., \cite{survey}), it follows that $v$ satisfies the differentiability and boundedness assumptions of the lemma.

We can now prove the main result of this paper.
Recall that
\begin{align*}
S_{a}u(x) &
= \Big(\int_{\Gamma_x^{(1,a)}(\gamma)}
\frac{1}{\gamma(B(y,t))}|t\nabla
e^{-t^{2}L}u(y)|^{2}\,d\gamma(y)\,\frac{dt}{t}\Big)^\frac12
\\ & = \Big(\int_{\R^n\times(0,\infty)}
\frac{1_{B(x,t)}(y)}{\gamma(B(y,t))}1_{(0,am(x))}(t)|t\nabla
e^{-t^{2}L}u(y)|^{2}\,d\gamma(y)\,\frac{dt}{t}\Big)^\frac12.
\end{align*}
It will be convenient to define, for $\e>0$,
\[
S_{a}^{\e}u(x) := \Big(\int_{\R^n\times(0,\infty)}
\frac{1_{B(x,t)}(y)}{\gamma(B(y,t))}1_{(\e,am(x))}(t)|t\nabla
e^{-t^{2}L}u(y)|^{2}\,d\gamma(y)\,\frac{dt}{t}\Big)^\frac12.\]

\begin{proof}[Proof of Theorem \ref{thm:main}]
As in the proof of Lemma \ref{lem:caccio} it suffices to consider real-valued
$u\in C_{\rm c}(\R^n)$.
Throughout the proof we fix $a > 0$ and set $K:=c_{a,1}$ and
$\tilde{K}:=c_{1+2K,2}$ using the notations of Lemma \ref{lem:admsym}.

Let $F\subseteq \R^n$ be an arbitrary closed set and define
$$
F^{*} := \big\{x \in \R^{n} \colon
\gamma(F\cap B(x,r)) \geq \tfrac{1}{2} \gamma(B(x,r))\ \forall r \in
(0,\tilde Km(x)] \;\big\}.
$$
Note that, since $F$ is closed, $F^* \subseteq F$.
For $0<\e<1$ and $1<\alpha<2$  put
$$
R_{\alpha} ^{\e} (F^{*}) := \, \{(y,t) \in \R^{n}\times (0,\infty) \colon
d(y,F^{*})<\alpha t \ \text{and}\ t \in (\a^{-1}\e,\a K m(y))\}$$
and let
$\partial R_{\alpha} ^{\e}(F^*)$ be its topological boundary.
As in \cite[page
162]{fs} and \cite[page 206]{littleStein}
we may regularise this set and thus assume it admits a surface measure
$d\sigma_\a^\e(y,t)$. Applying first Green's formula in $\R^n$ to the section of
$R_{\alpha} ^{\e}(F^{*})$
at level $t$ and using the definition of $L$ (see \eqref{eq:L}), and subsequently the fundamental theorem of
calculus in the $t$-variable,
 we obtain the estimate
\begin{align*}
\int  _{F^{*}} & |S_{a}^{\e}u(x)|^{2} \,d\gamma(x)
\\ & = \int_{\R^n \times (0,\infty)}\int_{F^*}
\frac{1_{B(x,t)}(y)}{\gamma(B(y,t))}1_{(\e,am(x))}(t)|t\nabla
e^{-t^{2}L}u(y)|^{2}\,d\gamma(x)\,d\gamma(y)\,\frac{dt}{t}
\\ & \stackrel{\rm(i)}{\le}
 \int_{\R^n \times (0,\infty)}\int_{F^*}
\frac{1_{B(y,t)}(x)}{\gamma(B(y,t))}1_{(\e,Km(y))}(t)|t\nabla
e^{-t^{2}L}u(y)|^{2}\,d\gamma(x)\,d\gamma(y)\,\frac{dt}{t}
\\ & \stackrel{\rm(ii)}=
\int_{\R^n \times (0,\infty)}
\frac{\gamma(B(y,t)\cap F^*))}{\gamma(B(y,t))}1_{\{d(y,F^*)< t\}}1_{(\e,Km(y))}(t)|t\nabla
e^{-t^{2}L}u(y)|^{2}\,d\gamma(y)\,\frac{dt}{t}
\\ & \le
\int_{\R^n \times (0,\infty)}
1_{\{d(y,F^*)< t\}}1_{(\e,Km(y))}(t)|t\nabla
e^{-t^{2}L}u(y)|^{2}\,d\gamma(y)\,\frac{dt}{t}
\\ & \leq  \int _{R_{\alpha} ^{\e}(F^{*})} |t\nabla e^{-t^{2}L}u(y)|^{2}\,
d\gamma(y)\,\frac{dt}{t} \\
& \lesssim \int _{R_{\alpha} ^{\e}(F^{*})} tLe^{-t^{2}L}u(y)
\cdot e^{-t^{2}L}u(y)\,d\g(y)\,dt
\\ & \qquad + \int _{\partial R_{\alpha} ^{\e}(F^{*})} |t\nabla
e^{-t^{2}L}u \cdot \nu^{/\!\!/}\!(y,t)||e^{-t^{2}L}u(y)|
e^{-\tfrac12 |y|^{2}}\,d\sigma_\a^\e(y,t)\\
&  \lesssim  \int _{R_{\alpha} ^{\e}(F^{*})}
 - \partial_{t}|e^{-t^{2}L}u(y)|^{2}\,d\gamma(y)\,dt
\\ & \qquad + \int _{\partial R_{\alpha} ^{\e}(F^{*})} |t\nabla
e^{-t^{2}L}u(y)||e^{-t^{2}L}u(y)|e^{-\tfrac12|y|^{2}}\,d\sigma_\a^\e(y,t)\\
& \lesssim \int _{\partial R_{\alpha} ^{\e}(F^{*})}
|e^{-t^{2}L}u(y)\nu^{\perp}(y,t)|^{2}
e^{-\tfrac12|y|^{2}}\,d\sigma_\a^\e(y,t)
\\ & \qquad + \int _{\partial R_{\alpha} ^{\e}(F^{*})} |t\nabla
e^{-t^{2}L}u(y)||e^{-t^{2}L}u(y)|e^{-\tfrac12|y|^{2}}\, d\sigma_\a^\e(y,t).
\end{align*}
In the above computation, $\nu^{/\!\!/}$ denotes the projection of the normal
vector $\nu$
to $R_{\alpha} ^{\e}$ onto $\R^{n}$ and $\nu^{\perp}$ the projection of $\nu$ in
the $t$ direction. In step (i) we used that
$1_{B(x,t)}(y) = 1_{B(y,t)}(x)$ and that
$|x-y|<t$ and $t<am(x)$ imply $t<Km(y)$ via Lemma \ref{lem:admsym}(i);
in step (ii) we used that $B(y,t)\cap F^*\not=\emptyset$ implies
$d(y,F^*)< t$.
Of course, all implied constants in the above inequalities are independent of
$F$, $\e$, $\a$, and $u$.

If $(y,t)\in \partial R_\a^\e(F^*)$, then either $d(y,F^*)=\alpha t$ and
$t\in [\alpha^{-1}\e, \a K m(y)]$,  or else $d(y,F^*)<\alpha t$ and
$t\in \{\alpha^{-1}\e, \a K m(y)\}$. By examining these three cases separately,
each time distinguishing between the possible relative positions of $m(y)$
with respect to the numbers $\frac12\e$, $\alpha^{-1}\e$, and  $\e$,
one checks that $\partial R_\a^\e(F^*)\subseteq  {\tilde B}^{\e} :=
\tilde{B}^{\e}_{1} \cup \tilde{B}^{\e}_{2} \cup\tilde{B}^{\e}_{3} $
with
\begin{align*}
\tilde{B}^{\e}_{1} & := \{(y,t) \in \R^{n} \times (0,\infty) \colon  t \in
[\tfrac12\e,\min\{\e,m(y)\}] \; \text{and} \; d(y,F^{*})\le 2t\}, \\
\tilde{B}^{\e}_{2} &:= \{(y,t) \in \R^{n} \times (0,\infty) \colon  t \in
[\e,m(y)] \; \text{and} \; t\le d(y,F^{*})\le 2t\}, \\
\tilde{B}^{\e}_{3} &:= \{(y,t) \in \R^{n} \times (0,\infty) \colon  t\in [
m(y),2Km(y)] \; \text{and} \; d(y,F^{*})\le 2t\}.
\end{align*}
Now notice that, on $\partial R_\a^\e(F^*)$, we have either
$t=\frac{\e}{\alpha}$, $t=\alpha Km(y)$, or $t=\alpha^{-1}{d(y,F^{*})}$.
Integrating over $\alpha \in (1,2)$ with respect to $\frac{d\alpha}{\alpha}$ and
changing variables using that $\frac{d\a}{\alpha} \sim
\frac{dt}{t}$, we obtain
\begin{align*}
\int  _{F^{*}}  |S_a^{\e}u|^{2}\,d\gamma
& \lesssim \int  _{{\tilde B}^{\e}} |e^{-t^{2}L}u(y)|^{2} d\gamma(y)\,\frac{dt}{t}
\\ & \quad+  \Big(\int  _{{\tilde B}^{\e}} |e^{-t^{2}L}u(y)|^{2}
d\gamma(y)\,\frac{dt}{t}\Big)^{\frac{1}{2}} \Big(\int  _{\tilde{B}^{\e}}
|t\nabla e^{-t^{2}L}u(y)|^{2} d\gamma(y)\,\frac{dt}{t}\Big)^{\frac{1}{2}}
\\ & \lesssim
\int  _{\tilde{B}^{\e}} |e^{-t^{2}L}u(y)|^{2} d\gamma(y)\,\frac{dt}{t}
+  \int  _{\tilde {B}^{\e}}
|t\nabla e^{-t^{2}L}u(y)|^{2} d\gamma(y)\,\frac{dt}{t}.
\end{align*}
Here, and in the estimates to follow, the implied constants
are independent of $F$, $\e$, and $u$.

We have to estimate the following six integrals:
\begin{align*}
I_{1} &:= \int  _{\tilde{B}^{\e}_{1}} |e^{-t^{2}L}u(y)|^{2}
d\gamma(y)\,\frac{dt}{t}, \qquad
I_{2} := \int  _{\tilde{B}^{\e}_{1}} |t\nabla e^{-t^{2}L}u(y)|^{2}
d\gamma(y)\,\frac{dt}{t}, \\
I_{3} &:= \int  _{\tilde{B}^{\e}_{2}} |e^{-t^{2}L}u(y)|^{2}
d\gamma(y)\,\frac{dt}{t}, \qquad
I_{4} := \int  _{\tilde{B}^{\e}_{2}} |t\nabla e^{-t^{2}L}u(y)|^{2}
d\gamma(y)\,\frac{dt}{t}, \\
I_{5} &:= \int  _{\tilde{B}^{\e}_{3}} |e^{-t^{2}L}u(y)|^{2}
d\gamma(y)\,\frac{dt}{t}, \qquad
I_{6} := \int  _{\tilde{B}^{\e}_{3}} |t\nabla e^{-t^{2}L}u(y)|^{2}
d\gamma(y)\,\frac{dt}{t}.
\end{align*}

We start with $I_{1}$ and remark that, for $(y,t) \in \tilde{B}^{\e}_{1}$, there
exists $x \in F^{*}$ such that $|x-y|\le 2t$.
Since $t\le \min\{\e,m(y)\}\le m(y)$, by Lemma \ref{lem:admsym}(i) we have $t\le
c_{1,2}m(x)$ and hence $t\le \tilde Km(x)$, noting that $c_{1,2} \leq c_{1 + 2 K, 2} = \tilde K$.
Therefore,
by the definition of $F^*$,
\begin{align}\label{eq:FcapB}
\gamma(F\cap B(x,t))\geq\tfrac{1}{2} \gamma(B(x,t)).
\end{align}
(At this point the reader may wonder why $F^*$ is defined in terms of $\tilde K$
and not in terms of $c_{1,2}$. The reason is that the argument will be repeated
in the estimation of $I_2$, $I_5$, and $I_6$; in the latter two cases, the definition
of $B^{\e}_3$ implies that one
only gets $t\le 2Km(y)$ and hence $t\le c_{2K,2}m(x) \leq c_{1+2K,2}m(x)$).
By \eqref{eq:FcapB} and doubling property for the admissible ball
$B(x,t)\in\mathscr{B}_{c_{1,2}}$,
 $$\gamma(F\cap B(y,3t)) \ge \gamma(F\cap B(x,t))\geq\tfrac{1}{2}
\gamma(B(x,t))\gtrsim \g(B(x,3t))\ge \g(B(y,t)),
 $$
and therefore
\begin{equation}\label{eq:vee}
\bal I_1
& \lesssim \int_{\tilde B_1^\e} \int_{F\cap B(y,3t)} \frac{1}{\g(B(y,t))}| e^{-t^2L}u(y)|^2
\,d\g(z)\,d\g(y)\,\frac{dt}{t}
\\ & \le \int_{\R^n}\int_{\frac12\e}^{\frac12\e\vee\min\{\e,m(y)\}}\int_{F}
 \frac{1_{B(y,3t)}(z)}{\g(B(y,t))}|e^{-t^2L}u(y)|^2 d\g(z)\,\frac{dt}{t}\,d\g(y)
\\ & \le
\int_{F}\int_{\frac12\e}^{\frac12\e\vee\min\{\e,c_{1,3}m(z)\}}\int_{B(z,3t)} \frac{1}{\g(B(y,t))}|
e^{-t^2L}u(y)|^2
\,d\g(y)\,\frac{dt}{t}\,d\g(z),
\eal
\end{equation}
where in the last inequality we used that $t\le m(y)$ and $|y-z|<3t$ imply $t\le
c_{1,3}m(z)$ by Lemma \ref{lem:admsym}(i).

Fix $(z,t)\in F\times
(\frac12\e,\frac12\e\vee\min\{\e,c_{1,3}m(z)\})$.
For all $y\in B(z,3t)$ we have $B(z,3t)\subseteq B(y,6t)$ and therefore, by the doubling property for $B(y,t)$
(noting that from $t<c_{1,3}m(z)$ and $|z-y|<3t$ it follows that
$t<c_{c_{1,3},3}m(y)$, so $B(y,t)$ is an admissible ball in
$\mathscr{B}_{c_{c_{1,3},3}}$),
$$
\bal
\int_{B(z,3t)}\frac{1}{\g(B(y,t))}| e^{-t^2L}u(y)|^2\,d\g(y)
& \lesssim \frac1{\g(B(z,3t))}\int_{B(z,3t)}| e^{-t^2L}u(y)|^2
\,d\g(y)
\\ & \le |T_{(3,c_{1,3})}^*u(z)|^2,
\eal$$
where the last inequality follows from
$(z,t)\in \Gamma_z^{(3,c_{1,3})}(\g)$.
Combining this with the previous inequality it
follows that
$$I_1  \lesssim
\int_{F}\int_{\frac12\e}^\e |T_{(3,c_{1,3})}^* u(z)|^2\,\frac{dt}{t}\,d\g(z)
\lesssim\int_{F}|T_{(3,c_{1,3})}^* u(z)|^2\,d\g(z).
$$

We proceed similarly for $I_{2}$, using Lemma \ref{lem:caccio}
to handle the gradient. With $\tau(z):= c_{1,3}m(z)$ we have, proceeding as in
\eqref{eq:vee},
\begin{align*}
I_{2} &
\lesssim  \int  _{F} \int  _{\frac12\e} ^{\frac12\e\vee\min\{\e,
\tau(z)\}}\int_{B(z,3t)} \frac{1}{\g(B(y,t))}|t\nabla
e^{-t^{2}L}u(y)|^{2}\,d\gamma(y)\,\frac{dt}{t}\,d\gamma(z)
 \\ & \stackrel{\rm(i)}{\lesssim }  \int  _{F\cap\{\tau(z)\ge \frac12\e\}} \int
_{\frac12\e}^\e
\frac{1}{\gamma(B(z,3\e))}\int_{B(z,3\e)} |t\nabla
e^{-t^{2}L}u(y)|^{2}\,d\gamma(y)\,\frac{dt}{t}\,d\gamma(z)
\\
 & \stackrel{\rm(ii)}{\lesssim}  \int  _{F\cap\{\tau(z)\ge \frac12\e\}}   \sum
_{l=2} ^{7}  \int_{\frac{l\e^{2}}{8}}^{\frac{(l+1)\e^{2}}{8}}
\frac{1}{\gamma(B(z,3\e))} \int
_{B(z,3\e)} |\nabla e^{-sL}u(y)|^{2}
\,d\gamma(y)\,ds\, d\gamma(z).
\end{align*}
In (i) we used the inclusions
$B(z,3t)\subseteq B(z,3\e)\subseteq B(z,6t)\subseteq B(y,9t)$
 together with the doubling property for $B(y,t)$, and
in (ii) we substituted $t^2 =s$.

For each $l\in \{2,\dots,7\}$
we apply Lemma \ref{lem:caccio} with $t_{0}^l =
\frac12(\frac{l\e^{2}}{8}+\frac{(l+1)\e^{2}}{8}) =
\frac{(2l+1)\e^{2}}{16},$ $c^l = 12$ and $(r^l)^2=\frac{\e^2}{16}$.
Together with the doubling property for $B(z,\e)$ (noting that $B(z,\e)\in
\mathscr{B}_{2c_{1,3}}$ in view of $\e\le2t\le 2c_{1,3}m(z)$), this
gives
\begin{align*}
I_{2}  &\lesssim \int_{F\cap\{\tau(z)\ge \frac12\e\}}  \sum   _{l=2} ^7
\int_{\frac{(2l-3)\e^2}{16}}^{\frac{(2l+5)\e^2}{16}}
\frac{1+r^l|z|}{(r^l)^{2}}\\ & \hskip4cm\times
\frac{1}{\gamma(B(z,6\e))}
\int  _{B(z,6\e)} |e^{-sL}u(y)|^{2}\,d\gamma(y)\,ds\,d\gamma(z).
\end{align*}

Fix $(z,s)\in (F\cap\{\tau(z)\ge \frac12\e\})\times
(\frac1{16}\e^2,\frac{19}{16}\e^2)$.
Then from $B(z,6\e)\subseteq B(z,24\sqrt s)\subseteq B(z,30\e)$
and the doubling property for the balls $B(z,\e) \in
\mathscr{B}_{2 c_{1,3}}$
(note that $\e \le 2\tau(z) = 2 c_{1,3} m(z)$),
\begin{align*}
\ & \frac{1}{\gamma(B(z,6\e))}
\int  _{B(z,6\e)} |e^{-sL}u(y)|^{2}\,d\gamma(y)
\\ & \qquad \lesssim \frac{1}{\gamma(B(z,24\sqrt s))}
\int  _{B(z,24\sqrt s)} |e^{-sL}u(y)|^{2}\,d\gamma(y)
 \le |T_{(24,4c_{1,3})}^*u(z)|^2,
\end{align*}
where the last step follows from $(z,\sqrt s)\in\Gamma_z^{(24,4c_{1,3})}(\g)$.
Combining this with the previous estimate we obtain
\begin{align*} I_{2}
& \lesssim   \int_F \sum   _{l=2} ^{7}
\int_{\frac{(2l-3)\e^2}{16}}^{\frac{(2l+5)\e^2}{16}}
\frac{1+r^l|z|}{(r^l)^{2}}
|T^{*}_{(24,4c_{1,3})}u(z)|^{2}\,ds\, d\gamma(z)
\\ & \lesssim \int  _{F}(1+\e|z|) |T^{*}_{(24,4c_{1,3})}u(z)|^{2}\,d\gamma(z),
\end{align*}
where the last step follows from the fact that
$r^l= \frac14\e$.

We proceed with an estimate for $I_{3}$.
Let $$G := \{y\in\R^n\colon 0< d(y,F^*)\le 2m(y)\}.$$
Using Lemma \ref{lem:cover2},
we cover $G$ with a sequence of balls $B(x_k,r_k)$
with $x_k\in G$ and $r_k = \tfrac14 d(x_k,F^*)$ for all $k$, and
\beq\label{eq:sumgamma}
\sum_{k\ge 1} \g(B(x_k, d(x_k,F^*))) \lesssim \g(G) \le \gamma(\complement F^*).
\eeq
with implied constant independent of $u$ and $F$. Note that $B(x_k,r_k)\in
\B_\frac12$ for all $k$.

If $(y,t)\in \tilde B_2^\e$, then $y\in G$ and therefore $y\in B(x_k,r_k)$ for
some $k$, and $\frac12d(y,F^*)\le t\le d(y,F^*)$.
It follows that
\beq\label{eq:I3}
\bal I_{3} & \leq \sum  _{k}
\int  _{B(x_k,r_{k})} \int^{d(y,F^*)}_{\frac12d(y,F^*)}
|e^{-t^{2}L}u(y)|^{2}\,\frac{dt}{t}\,d\gamma(y)
\\ & \leq \sum  _{k}
\int  _{B(x_k,r_{k})} \int^{\frac54d(x_k,F^*)}_{\frac14d(x_k,F^*)}
|e^{-t^{2}L}u(y)|^{2}\,\frac{dt}{t}\,d\gamma(y)
\\ & \leq \sum  _{k}
 \int^{\frac54d(x_k,F^*)}_{\frac14d(x_k,F^*)}\int  _{B(x_k,t)}
|e^{-t^{2}L}u(y)|^{2}\,d\g(y)\,\frac{dt}{t}.
\eal
\eeq
In the second inequality we used that $y\in B(x_k,r_k)$ implies $|x_k-y|<r_k=
\frac14 d(x_k,F^*)$, and the third inequality follows from Fubini's theorem and
the inequality $r_k=\frac14d(x_k,F^*) \leq \frac12d(y,F^*)\leq t$.

Fix an index $k$ and a number
$t\in (\frac14 d(x_k,F^*),\frac54 d(x_k,F^*))$. Since $F^*$ is contained in
the closure of $F$ we may pick $z_k\in F$
such that $|x_k-z_k|<2d(x_k,F^*).$
By the choice of $t$ this implies $|x_k-z_k|< 8t.$
Since by assumption we have $t\le \frac54 d(x_k,F^*) \le \frac52m(x_k)$ (the
second inequality being a consequence of $x_k\in G$),
and since $|x_k-z_k|<8t$, from Lemma \ref{lem:admsym} we conclude that $t\le
dm(z_k)$ with $d:=c_{\frac52,8}$.
We conclude that $(x_k,t)\in \Gamma_{z_k}^{(8,d)}(\g)$  (since by definition
this means that
$|x_k-z_k|\le 8t\le 8d m(z_k)$)
and consequently, using the doubling property for the admissible ball
$B(x_k,t)\in \mathscr{B}_{\frac52}$,
$$
\bal
\ & \frac1{\g(B(x_{k},t))} \int  _{B(x_{k},t)} |e^{-t^{2}L}u(y)|^{2}\,d\gamma(y)
\\ & \qquad \lesssim\frac1{\g(B(x_{k},8t))} \int  _{B(x_{k},8t)}
|e^{-t^{2}L}u(y)|^{2}\,d\gamma(y) \le |T_{(8,d)}^* u(z_k)|^2.
\eal$$
Combining this with the previous inequalities we obtain
\begin{align*}
I_3 & \lesssim \Big(\sup_{z\in F}|T_{(8,d)}^* u(z)|^2\Big)\sum_k
\int^{\frac54d(x_k,F^*)}_{\frac14d(x_k,F^*)}\g(B(x_k,t)) \,\frac{dt}{t}
 \\
& \lesssim \Big(\sup_{z\in F}|T_{(8,d)}^* u(z)|^2\Big)\sum  _{k}
\gamma(B(x_{k},\tfrac54 d(x_k,F^*))) \\
& \lesssim \Big(\sup_{z\in F}|T_{(8,d)}^* u(z)|^2\Big) \gamma(\complement F^*),
\end{align*}
where the last step used \eqref{eq:sumgamma} and the doubling property (recall
that $d(x_k,F^*)\le 2m(x_k)$, so the balls $B(x_k, d(x_k,F^*))$ belong to
$\mathscr{B}_2$).

For estimating $I_{4}$, we let $G$ and $B(x_k,r_k)$ be as in the previous estimate.
Proceeding as in the first two lines of \eqref{eq:I3} and applying the Fubini
theorem, we get
\begin{align*}
I_{4} &\lesssim \sum  _{k}
 \int^{\frac54 d(x_k,F^*)}_{\frac14d(x_k,F^*)}\int  _{B(x_{k},r_k)} |t\nabla
e^{-t^{2}L}u(y)|^{2}\,d\g(y)\,\frac{dt}{t}
 \\ & = \frac12\sum  _{k} \sum_{l=2}^{49}
 \int^{\frac{2l+2}{64} d^2(x_k,F^*)}_{\frac{2l}{64}d^2(x_k,F^*)}\int
_{B(x_{k},r_k)} |\nabla e^{-sL}u(y)|^{2}\,d\gamma(y)\,ds.
\end{align*}
By Lemma \ref{lem:caccio}, applied with $t_0 = \frac{2l+1}{64} d^2(x_k,F^*),$
$c=2$
and $r = \frac{1}{8} d(x_k,F^*)$, this gives the estimate
\begin{align*}
I_{4}
&\lesssim \sum_k \sum_{l=2}^{49}
 \int^{\frac{2l+5}{64} d^2(x_k,F^*)}_{\frac{2l-3}{64}d^2(x_k,F^*)}
\frac{1+d(x_k,F^*)|x_{k}|}{d^2(x_k,F^*)}
 \int  _{B(x_{k},\frac12 d(x_k,F^*))}|e^{-sL}u(y)|^{2}\,d\gamma(y)\,ds\\
 &\le \sum_k \sum_{l=2}^{49}
 \int^{\frac{2l+5}{64} d^2(x_k,F^*)}_{\frac{2l-3}{64}d^2(x_k,F^*)}
\frac{3}{d^2(x_k,F^*)}
 \int  _{B(x_{k},4\sqrt{s})}|e^{-sL}u(y)|^{2}\,d\gamma(y)\,ds,
\end{align*}
where we used that $d(x_k,F^*)\le 2m(x_k)\le \frac2{|x_k|}$ and that $s\ge
\frac1{64}d^2(x_k,F^*)$ implies $\frac12 d(x_k,F^*)\le 4\sqrt s$.

Fix $k$ and pick an element $z_k\in F$ such that
$|x_k-z_k|<2d(x_k,F^*)$. Then for all $s$ in the range of integration we have
$|x_k-z_k|<16\sqrt s$.
Since $\sqrt s \le \frac32 d(x_k,F^*) \le 3 m(x_k)$, from Lemma \ref{lem:admsym}
we conclude that $\sqrt s\le dm(z_k)$ with $d:=c_{3,16}$.
We conclude that $(x_k,4\sqrt s)\in \Gamma_{z_k}^{(4,4d)}(\g)$.
This gives
\begin{align*}I_{4}&\lesssim
\Big( \sup_{z\in F} |T_{(4,4d)}^* u(z)|^2\Big)\sum  _k
\frac{1}{d^2(x_k,F^*)}\int_{\frac1{64} d^2(x_k,F^*)}^{\frac{103}{64}
d^2(x_k,F^*)}\g(B(x_k,4\sqrt s)\,ds
\\ & \lesssim  \Big(\sup_{z\in F} |T_{(4,4d)}^* u(z)|^2\Big)\sum_k
\g(B(x_k,\tfrac12\sqrt{103}d(x_k,F^*)))
\\ & \lesssim  \Big(\sup_{z\in F} |T_{(4,4d)}^* u(z)|^2\Big)\sum_k
\g(B(x_k,d(x_k,F^*)))
\\ & \lesssim \Big(\sup_{z\in F} |T_{(4,4d)}^* u(y)|^2\Big)\gamma(\complement
F^*),
\end{align*}
where the second last step used the doubling property for admissible balls
(recalling that $B(x_k,d(x_k,F^*))\in \mathscr{B}_2$), and the last step used
\eqref{eq:sumgamma}.

To estimate $I_{5}$, we proceed as we did for $I_{1}$:
\begin{align*}
I_5 & \lesssim \int_{\tilde B_3^\e}\int_{F\cap B(y,3t)} \frac{1}{\g(B(y,t))}| e^{-t^2L}u(y)|^2
\,d\g(z)\,d\g(y)\,\frac{dt}{t}
\\ & \le \int_{\R^n}\int_{m(y)}^{2Km(y)}\int_{F} \frac{1_{B(y,3t)}(z)}{\g(B(y,t))} |
e^{-t^2L}u(y)|^2
\,d\g(z)\,\frac{dt}{t}\,d\g(y)
\\ & \stackrel{\rm(i)}{\le}
\int_{F}\int_{(1+3c_{2K,3})^{-1}m(z)}^{c_{2K,3}m(z)}\int_{B(z,3t)} \frac{1}{\g(B(y,t))}
|e^{-t^2L}u(y)|^2 \,d\g(y)\,\frac{dt}{t}\,d\g(z)
\\ & \lesssim \int_{F}|T_{(2K,c_{2K,3})}^* u(z)|^2\,d\g(z),
\end{align*}
where in step (i) we used that $m(y)\le t\le 2Km(y)$ and $|y-z|<3t$ imply $t\le
c_{2K,3}m(z)$ by Lemma \ref{lem:admsym}(i), so $|y-z|<3c_{2K,3}m(z)$, and by an
application of Lemma \ref{lem:admsym}(ii) the latter implies
$m(z)\le (1+3c_{2K,3})m(y)\le (1+3c_{2K,3})t $.

Finally we turn to $I_{6}$, which is treated as $I_{2}$. With $c=c_{2K,3}$ and $d
= (1+3c_{2K,3})^{-1}$ as in the previous estimate, and using Lemma
\ref{lem:caccio} as in the estimate for $I_2$, we get
$$
\bal
I_6
& \lesssim   \int  _{F} \int  _{dm(z)}
^{cm(z)}\frac{1}{\gamma(B(z,3t))}\int_{B(z,3t)} |t\nabla
e^{-t^{2}L}u(y)|^{2}\,d\gamma(y)\,\frac{dt}{t}\,d\gamma(z)
\\ & = \frac12   \int  _{F} \int  _{d^2m(z)^2}
^{c^2m(z)^2}\frac{1}{\gamma(B(z,3t))}\int_{B(z,3t)} |\nabla
e^{-sL}u(y)|^{2}\,d\gamma(y)\,ds\,d\gamma(z)
\\ & \lesssim \int_{F} (1+m(z)|z|) |T_{(\rr,\al)}^* u(z)|^2\,d\g(z)
\\ & \lesssim \int_{F} |T_{(\rr,\al)}^* u(z)|^2\,d\g(z),
\eal
$$
for certain $\rr,\al$ independent of $u$, $F$, and $\e$.

Combining all these estimates, we obtain six couples $(\rr^{(j)},\al^{(j)})$
($j=1,...,6$),
and, passing to the limit $\e\downarrow 0$, the following estimate,
valid for arbitrary closed subsets $F\subseteq \R^n$:
\beq\label{eq:Su}
\bal
\ & \int  _{F^{*}} |S_{a}u(x)|^{2}\,d\gamma(x) \\ &
\qquad \lesssim
\sum  _{j=1} ^{6} \Big(
\big(\sup_{z\in F}
|T^{*}_{(\rr^{(j)},\al^{(j)})}u(z)|^{2}\big)\gamma(\complement F^{*})+ \int  _{F}
|T^{*}_{(\rr^{(j)},\al^{(j)})}u(z)|^{2}\,d\gamma(z) \Big),
\eal
\eeq
with constants independent of $F$ and $u$.

To finish the proof, we consider the distribution functions
$$
\bal\gamma_{S_{a}u}(\sigma)
& :=\gamma\big(\big\{x \in \R^{n}\colon S_{a}u(x)>\sigma\big\}\big), \\
 \gamma_{T^{*}_{(\rr^{(j)},\al^{(j)})}u}(\sigma) &:=\gamma\big(\big\{x \in
\R^{n}\colon
T^{*}_{(\rr^{(j)},\al^{(j)})}u(x)>\sigma\big\}\big), \quad j=1,\dots,6.
\eal
$$
We fix $\sigma>0$ for the moment, and apply \eqref{eq:Su} to the set
$$F_\sigma := \big\{z\in\R^n\colon T_{(\rr^{(j)},\al^{(j)})}^* u(z)\le \sigma, \
j=1,\dots,6\big\},$$
and claim that
$\complement F_\sigma^{*}
\subseteq \{M_{\tilde K}^{*}(1_{\complement F_\sigma}) > \frac{1}{2}\}$.
Indeed, let $x\in \complement F_\sigma^{*}$ and fix $r \in (0,\tilde Km(x)]$ such that
$\gamma(B(x,r)\cap F_\sigma)<\frac12\gamma(B(x,r))$. Then
$$M_{\tilde K}^{*}(1_{\complement F_\sigma})(x) \ge \frac{\gamma(B(x,r)\cap
\complement F_\sigma)}{\gamma(B(x,r))} > \frac{1}{2},$$
proving the claim.

Lemma \ref{lem:maximal} (with admissibility parameter $\tilde K$, $\tau = \frac12$, applied to the function $1_{\complement F_\sigma}$)
gives us
$\gamma(\complement F_\sigma^{*})
\lesssim
\gamma(\complement F_\sigma)$.
Using this in combination with the definition of $F_\sigma$, for $j=1,\dots,6$ we obtain
\begin{align*}
 \frac1{\sigma^2}\big(\sup_{z\in F_\sigma} |T^{*}_{\rr^{(j)},\al^{(j)}
}u(z)|^{2}\big)\gamma(\complement F_\sigma^{*})
& \le \g(\complement F_\sigma^*) \lesssim \g(\complement F_\sigma) \le \sum_{k=1}^6
 \g\big(\big\{T^{*}_{(\rr^{(k)},\al^{(k)})}u>\sigma\big\}\big).
\end{align*}
Hence, from \eqref{eq:Su} we infer
$$
\bal
\gamma_{S_{a}u}(\sigma)
& \le \gamma(F_\sigma^{*}\cap\{S_{a}u > \sigma\}) +  \gamma(\complement F_\sigma^{*}) \\
&\lesssim \frac{1}{\sigma^{2}} \int  _{F_\sigma^{*}} |S_{a}u(x)|^{2}\,d\gamma(x) +
\gamma(\complement F_\sigma)
\\ & \lesssim \sum_{j=1}^6 \Big[\gamma_{T_{(\rr^{(j)}, \al^{(j)})}^*u}(\sigma) +
\frac{1}{\sigma^{2}}
\int  _{F_\sigma}
|T^{*}_{(\rr^{(j)},\al^{(j)})}u(z)|^{2}\,d\gamma(z)\Big]
\\ & \lesssim  \sum_{j=1}^6 \Big[ \gamma_{T_{(\rr^{(j)},\al^{(j)})}^*u}(\sigma) +
\frac{1}{\sigma^{2}} \int  _{0}
^{\sigma} t\gamma_{T_{(\rr^{(j)},\al^{(j)})}^*u}(t)\,dt\Big].
\eal
$$
Integrating over $\sigma$ and noting that
$$
\bal \int_0^\infty \frac{1}{\sigma^{2}} \int  _{0} ^{\sigma}
t\gamma_{T_{(\rr^{(j)},\al^{(j)})}^*u}(t)\,dt\,d\sigma
& = \int_0^\infty t\gamma_{T_{(\rr^{(j)},\al^{(j)})}^*u}(t)\int_t^\infty
\frac1{\sigma^2}
\,d\sigma\,dt
\\ & = \int_0^\infty \gamma_{T_{(\rr^{(j)},\al^{(j)})}^*u}(t)\,dt = \big\n
T_{(\rr^{(j)},\al^{(j)})}^*u\big\n_{L^1(\g)},
\eal$$
 we get, by Theorem \ref{thm:aperture} and
with $\al^{(j)}{}'$ as in the statement of that theorem,
\begin{align*}
\|S_{a}u\|_{L^1(\g)}
 \lesssim \sum_{j=1}^6 \big\|T^{*}_{(\rr^{(j)},\al^{(j)})}u\big\|_{L^1(\g)}
 \lesssim \sum_{j=1}^6
\big\|T^{*}_{(1,\al^{(j)}{}')}u\big\|_{L^1(\g)}
\le 6\big\|T^{*}_{(1,\al')}u\big\|_{L^1(\g)},
\end{align*}
where $\al' = \displaystyle\max_{j=1,\dots,6} \al^{(j)}{}'$.
\end{proof}

\begin{remark}\label{rem:}
In \cite{mnp2} the cones
$$\widetilde \Gamma^{(\rr,\al)} _{x}(\gamma) := \big\{(y,t) \in \R^{n}\times(0,\infty)\colon
|y-x| < \rr t  \textrm{ and }  t < \al  m(y)\big\}$$
are used implicitly. In view of the inclusions
$$ \Gamma^{(\rr,\al)} _{x}(\gamma)
\subseteq \widetilde \Gamma^{(\rr,c_{a,A})} _{x}(\gamma), \qquad
 \widetilde \Gamma^{(\rr,\al)} _{x}(\gamma)
\subseteq \Gamma^{(\rr,c_{a,A})} _{x}(\gamma),
$$
the corresponding functions  $\widetilde S$ and $\widetilde T^*$
satisfy the pointwise bounds
$$\widetilde S_{a}u(x) \lesssim S_{c_{a,A}} u(x), \qquad
S_{a} u(x) \lesssim \widetilde S_{c_{a,A}}u(x)$$
and
$$\widetilde T_{(A,a)}^*u(x) \lesssim T_{(A,c_{a,A})}^*u(x), \qquad
T_{(A,a)}^* u(x) \lesssim \widetilde T_{(A,c_{a,A})}^* u(x).$$
In particular, Theorem \ref{thm:main} remains valid if we replace $S$ and $T^*$ by $\widetilde S$ and $\widetilde T^*.$
Remark also that \cite[Theorem 3.8]{mnp2} gives a change of aperture formula for tent spaces that implies an analogue of Theorem \ref{thm:aperture} for the square function $\tilde{S}$ for $A, A' > 1$.
\end{remark}

\begin{remark}
We conclude with a few words on reverse inequalities, i.e., controls of the maximal function by the square function.
In the euclidean case, such inequalities are generally proven via atomic decompositions, usually going through tent spaces. We have developed, in \cite{mnp2}, the gaussian analogues of these spaces and their atomic decomposition. However, to deduce a reverse inequality, we would then need an adequate analogue of the Calder\'on reproducing formula (analogues exist, but do not seem to be appropriate), and such a formula is involving all $t\in(0,\infty)$ rather than just $t\in(0,a m(x))$. A complete Hardy space theory is thus likely to require an understanding of the ``non-admissible" parts of objects such as $T^{*}u$, $Su$, or Mauceri-Meda's atoms (i.e. the part corresponding to the scales $t \in (am(x),\infty)$, for which balls are not admissible), or a technique that allows one to avoid such non-admissible part in arguments involving Calder\'on reproducing formulae.
This is the subject of some of our on-going investigations.
\end{remark}

\end{document}